\documentclass[a4paper,10pt,reqno]{amsart}

\title{Quasi-arithmetic means ad libitum}

\usepackage[T1]{fontenc}
\usepackage{amsmath}
\usepackage{amssymb}
\usepackage{amsthm}
\usepackage[left=3cm, right=3cm]{geometry}
\usepackage{hyperref}
\usepackage{fancyhdr}
\usepackage{enumitem}
\usepackage{comment}
\usepackage{nicefrac}
\usepackage{bm}
\usepackage{mathrsfs}
\usepackage{graphicx}
\usepackage[utf8]{inputenc}
\usepackage{cancel}
\usepackage{mathtools}
\usepackage{tikz}
\usetikzlibrary{decorations.pathreplacing,matrix,arrows,positioning,automata,shapes,shadows,calc,fadings,decorations,snakes,through,intersections}
\usepackage{float}

\AtBeginDocument{%
   \def\MR#1{}
}

\newtheorem{thm}{Theorem}[section]
\newtheorem{cor}[thm]{Corollary}%[section]
\newtheorem{lem}[thm]{Lemma}

\theoremstyle{definition} 
\newtheorem{defi}[thm]{Definition}%[section]
\let\olddefi\defi
\renewcommand{\defi}{\olddefi\normalfont}
\newtheorem{question}{Question}%[section]
\let\oldquestion\question
\renewcommand{\question}{\oldquestion\normalfont}
\newtheorem{example}[thm]{Example}
\let\oldexample\example
\renewcommand{\example}{\oldexample\normalfont}
\newtheorem{rmk}[thm]{Remark}
\let\oldrmk\rmk
\renewcommand{\rmk}{\oldrmk\normalfont}
\newtheorem{claim}{\textsc{Claim}}

{\textsc{Paolo Leonetti}}
\hypersetup{
    pdftitle={Quasi-arithmetic means ad libitum},
    pdfauthor={},
    pdfmenubar=false,
    pdffitwindow=true,
    pdfstartview=FitH,
    colorlinks=true,
    linkcolor=blue,
    citecolor=green,
    urlcolor=cyan
}

\providecommand{\MR}[1]{}

\providecommand{\MR}{\relax\ifhmode\unskip\space\fi MR }

\providecommand{\href}[2]{#2}

\makeatletter
\@namedef{subjclassname@1991}{JEL code}
\makeatother
\begin{document}

\author[P.~Leonetti]{Paolo Leonetti}
\address
%[Paolo~Leonetti]
{Department of Economics, Universit\`a degli Studi dell'Insubria, via Monte Generoso 71, Varese 21100, Italy}
%{Department of Statistics, Universit\`a Luigi Bocconi, via Roentgen 1, Milan 20136, Italy}
\email{paolo.leonetti@uninsubria.it}
\urladdr{\url{https://sites.google.com/site/leonettipaolo/}} 


\begin{abstract} 
\noindent 
Let $\alpha_1, \ldots, \alpha_m$ be two or more positive reals with sum $1$, let $C\subseteq \mathbb{R}^k$ be an open convex set, and $f: C\to \mathbb{R}^k$ be a continuous injection with convex image. 
For each nonempty set $S\subseteq C$, let $\mathscr{M}(S)$ be the family of quasi-arithmetic means of all $m$-tuples of vectors in $C$ with respect to 
$f$ and the weights $\alpha_1,\ldots,\alpha_m$, that is, the family  
$$
\mathscr{M}(S)= \left\{ f^{-1}\left(\alpha_1f(x_1)+\cdots+\alpha_mf(x_m)\right): x_1,\ldots,x_m \in S \right\}.
$$
We provide a simple necessary and sufficient condition on $S$ for which the infinite iteration $\bigcup_{n}\mathscr{M}^n(S)$ is relatively dense in the convex hull of $S$. 
\end{abstract}
\maketitle
\thispagestyle{empty}

\section{Introduction}\label{sec:int}

Let $C$ be a nonempty convex subset of a real vector space $X$ and let $f: C\to X$ be an injective map with convex image.  
Let also $\alpha_1,\ldots,\alpha_m$ be two or more positive reals with sum $1$. Then the \emph{quasi-arithmetic mean} of $x_1,\ldots,x_m\in C$ with respect to the function $f$ and the weights $\alpha_1,\ldots,\alpha_m$ is the \textquotedblleft mean vector\textquotedblright
$$
\mathsf{m}(x_1,\ldots,x_m):=
f^{-1}\left(\alpha_1f(x_1)+\cdots+\alpha_mf(x_m)\right),
$$ 
so that $\mathsf{m}$ is a map $C^m \to C$.

Quasi-arithmetic means, which are commonly known also as quasi-linear means, provide a generalization of classical and weighted means (say, the arithmetic mean, the quadratic mean, the harmonic mean, and others). 
They were first considered by Kolmogorov \cite{Kol30}, Nagumo \cite{Nag30}, de Finetti \cite{deFin31}, and Kitagawa \cite{Kit34}, and have been proved to be useful in a large variety of contexts, 
see e.g. 
\cite[Section 7.3 and Chapters 15, 17, and 20]{MR1004465} and \cite[Chapter 5]{MR4274074}. 
Their lattice structure has been recently studied in \cite{MR4057519}. 
In addition, the investigation of functional equations involving quasi-arithmetic means dates back at least to Aumann \cite{MR1581521} and it has been the subject of intense research for about eighty years, see, e.g., the textbook expositions \cite[Chapter III]{MR0046395} and \cite[Chapters 15 and 17]{MR1004465}, the articles 
\cite{MR2176021, 
Glaz2023, 
MR3720973, %RIMETTERE
MR1487253, 
MR3542948, %RIMETTERE
MR2154023, MR2217857}, 
%, MR900699}, 
and references therein; in turn, the latter works revealed connections with the study of fixed point theory, permutable mappings, and dynamical systems.

For each nonempty $S\subseteq C$, we write 
$\mathscr{M}(S)$ 
for the set of all quasi-arithmetic means of vectors in $S$ with respect to $f$ and $\alpha_1,\ldots,\alpha_m$, that is, 
$$
\mathscr{M}(S):=\left\{\mathsf{m}(x_1,\ldots,x_m): x_1,\ldots,x_m \in S\right\}. 
$$
Considering that $\mathsf{m}(x,\ldots,x)=x$ for all $x \in C$, $\mathscr{M}(\cdot)$ is a monotone increasing operator (i.e., $S\subseteq \mathscr{M}(S)$ for all $S\subseteq C$), so it makes sense to define
$$
\mathscr{M}^\omega(S):=\bigcup_{n\ge 0}\mathscr{M}^n(S),
$$
where $\mathscr{M}^0(S):=S$ and $\mathscr{M}^{n+1}(S):=\mathscr{M}(\mathscr{M}^n(S))$ for all integers $n\ge 0$ and sets $S\subseteq C$.  
Of course, each $\mathscr{M}^n(S)$ is a subfamily of finite convex linear combinations of elements of $S$, hence the limit $\mathscr{M}^\omega(S)$ is contained in the convex hull of $S$, hereafter denoted by $\mathrm{conv}(S)$. 

\begin{rmk}\label{rmk:Momega}
Note that $\mathscr{M}^\omega(S)$ is the smallest set $W\subseteq C$ containing $S$ such that $\mathscr{M}(W)=W$. 
For, pick $x_1,\ldots,x_m \in \mathscr{M}^\omega(S)$. Then there exists an integer $n\ge 0$ such that $x_1,\ldots,x_m \in \mathscr{M}^n(S)$. Hence the quasi-arithmetic mean $\mathrm{m}(x_1,\ldots,x_m)$ belongs to $\mathscr{M}^{n+1}(S)$, with the conclusion that $\mathscr{M}(\mathscr{M}^{\omega}(S))=\mathscr{M}^\omega(S)$. 
\end{rmk}

As a related notion, in the case where $f$ is the identity function $\mathrm{id}$, following Mesikepp \cite{MR3508484}, the set $\mathscr{M}(S)$ can be regarded as the Minkowski $M$-sum 
$$
\oplus_M (S,\ldots,S):=\{a_1x_1+\cdots+a_mx_m: x_1,\ldots,x_m \in S, (a_1,\ldots,a_m) \in M\},
$$
where $M$ is the singleton $\{(\alpha_1,\ldots,\alpha_m)\}$; cf. also \cite{MR3581298, MR3120744}. 

\emph{Here and after, we suppose that $X=\mathbb{R}^k$, for some integer $k\ge 1$, and that $C$ is a nonempty open convex set.}  
Let us also denote by $\mathrm{int}(S)$ and $\mathrm{cl}(S)$ the interior and the closure of a subset $S\subseteq \mathbb{R}^k$ in the relative topology of $C$, respectively. A subset $S\subseteq \mathbb{R}^k$ is said to be $k$-dimensional whenever it does not lie in an affine $(k-1)$-dimensional hyperplane, or, equivalently, $\mathrm{int}(\mathrm{conv}(S))$ is nonempty. Every subset of $\mathbb{R}^k$ is endowed with its relative topology. Hence, given subsets $A, B\subseteq \mathbb{R}^k$, we say that $A$ is relatively dense in $B$ if the closure of $A\cap B$ in the relative topology of $B$ coincides with $B$. Lastly, let $|S|$ denote the cardinality of a set $S$. 

In the case where $f$ is the identity function and $C=\mathbb{R}^k$, the sets $\mathscr{M}^\omega(S)$ have been studied by Green and Gustin in \cite{MR42142}. 
Remarkably, they proved in \cite[Theorem 2.2]{MR42142} the following density result.
\begin{thm}\label{thm:originalgreengustin}
Suppose that $f$ is the identity function $\mathrm{id}$ on $\mathbb{R}^k$ and fix two or more positive weights $\alpha_1,\ldots,\alpha_m$ with sum $1$. Then $\mathscr{M}^\omega(S)$ is relatively dense in $\mathrm{conv}(S)$ for every set $S\subseteq \mathbb{R}^k$. 
\end{thm}

Theorem \ref{thm:originalgreengustin} is the starting point of this work. Our main question is to understand which continuous injections $f$ may replace the identity map $\mathrm{id}$ in the above result.

\section{Main Result}\label{sec:mainresultmain}
Our main result follows. 
\begin{thm}\label{thm:main}
Let $C\subseteq \mathbb{R}^k$ be an open convex set and $f: C\to \mathbb{R}^k$ be a continuous injection with convex image. Fix also two or more positive weights $\alpha_1,\ldots,\alpha_m$ with sum $1$, and a $k$-dimensional subset $S\subseteq C$. Then the following are equivalent\textup{:}
\begin{enumerate}[label={\rm (\roman{*})}]
    \item \label{item:cond1} $\mathscr{M}^\omega(S)$ is relatively dense in $\mathrm{conv}(S)$\textup{;}
    \item \label{item:cond2} $f[\mathrm{int}(\mathrm{conv}(S)]\subseteq \mathrm{int}(\mathrm{conv}(f[S]))$\textup{.}
\end{enumerate}
If, in addition, $S$ is compact, then they are also equivalent to\textup{:}
\begin{enumerate}[label={\rm (\roman{*})}]
\setcounter{enumi}{2}
    \item \label{item:cond3} $f[\mathrm{conv}(S)]\subseteq \mathrm{conv}(f[S])$\textup{.}
\end{enumerate}
\end{thm}
Before we proceed to the proof of Theorem \ref{thm:main}, some remarks are in order. First of all, if $C=\mathbb{R}^k$ and $f=\mathrm{id}$, then condition \ref{item:cond2} clearly holds, hence we recover Green and Gustin's Theorem \ref{thm:originalgreengustin} 
(in the case where $S$ is not $k$-dimensional, it is enough to consider, modulo homeomorphisms, the problem replacing $\mathbb{R}^k$ with the affine hull of $S$; we omit details). 

Second, the hypothesis that $f$ has convex image is not automatically satisfied if $k\ge 2$: for instance, consider the continuous injection $f: C\to \mathbb{R}^2$ with 
$C:=(0,1)^2$
defined by $f(x,y):=(x,x^2+y)$ for every $(x,y) \in C$. 

Lastly, by condition \ref{item:cond3}, one may think that there exists some relationship between $f[\mathrm{conv}(S)]$ and $\mathrm{conv}(f[S])$, apart from the trivial inclusion 
$$
f[S] \subseteq f[\mathrm{conv}(S)] \cap \mathrm{conv}(f[S]).
$$
But this is not the case: we may have equality, for instance, if $k=2$, $C=(0,\infty)^2$, $S=\{(1,1),(2,2)\}$, and $f: C\to \mathbb{R}^2$ is the continuous injection defined by $f(x,y):=(x,x^2+y^2)$ for all $(x,y) \in C$. Indeed, $f[\mathrm{conv}(S)]=\{(x,2x^2): x \in [1,2]\}$ and $\mathrm{conv}(f[S])$ is the line segment between $(1,2)$ and $(2,4)$.

\begin{rmk}\label{rmk:nonequivalence} 
The equivalence between \ref{item:cond1} and \ref{item:cond3} does not hold without any conditions on $S$. For, suppose that $m=k=2$, $\alpha_1=\alpha_2=1/2$, and let $f: \mathbb{R}^2\to \mathbb{R}^2$ be an homeomorphism with the property that $f[D]=B$ and $f(1,0)=(1,0)$, where $D:=[-1,1]^2$ and $B$ is the closed unit ball $B$ with center $(0,0)$ and radius $1$ 
(for instance, it is sufficient to define $f(0,0):=(0.0)$ and $f(x,y):=(\max\{|x|,|y|\}\cdot \cos(\theta), \max\{|x|,|y|\}\cdot \sin(\theta))$ otherwise, where $\theta$ is the unique angle in $[0,2\pi)$ such that $x=k\cos(\theta)$ and $y=k\sin(\theta)$ for some $k \in \mathbb{R}$). 
Then, setting $S:=D\setminus \{(1,0)\}$, it follows that $\mathscr{M}^\omega(S)$ is relatively dense in $\mathrm{conv}(S)=D$, hence condition \ref{item:cond1} holds. On the other hand, $f[\mathrm{conv}(S)]=B$ is not contained in $\mathrm{conv}(f[S])=B\setminus \{(1,0)\}$, hence condition \ref{item:cond3} fails. 
\end{rmk}

As a consequence, we obtain an extension of Green and Gustin's Theorem \ref{thm:originalgreengustin} in the case of boxes $S^k$; here, an interval $I\subseteq \mathbb{R}$ may be unbounded. 
\begin{cor}\label{cor:greengustin}
Let $f_1,\ldots,f_k: I \to \mathbb R$ be continuous injections such that $I\subseteq \mathbb{R}$ is an open interval. Define the map $f: I^k \to \mathbb{R}^k$ by 
$$
\forall x \in I^k, \quad 
f(x):=(f_1(x_1),\ldots,f_k(x_k))
$$ 
Fix also two or more positive weights $\alpha_1,\ldots,\alpha_m$ with sum $1$. 
Then $\mathscr{M}^\omega(S^k)$ is relatively dense in $\mathrm{conv}(S^k)$ for each $S\subseteq I$ with nonempty interior. 
\end{cor}

We will need a series of intermediate lemmas given in Section \ref{sec:intermediate}. We will deduce Theorem \ref{thm:main} and Corollary \ref{cor:greengustin} in Section \ref{sec:proofs}. 

%%%%%%%%%%%%%%%%%%%%%%%%%%%%%%%%%%%%%%%%%%%%%%%%%%%%%%%%%%
%%%%%%%%%%%%%%%%%%%%%%%%%%%%%%%%%%%%%%%%%%%%%%%%%%%%%%%%%%%

\section{Intermediate Lemmas}\label{sec:intermediate}

In what follows, $C\subseteq \mathbb{R}^k$ stands for an open convex set, $f: C\to \mathbb{R}^k$ for a continuous injection with convex image, and $S$ for a $k$-dimensional subset of $C$ (as in the statement of Theorem \ref{thm:main}).

\begin{lem}\label{lem:brower}
The set $f[C]$ is open in $\mathbb{R}^k$ and the function $f$ is an homeomorphism onto its image. 
\end{lem}
\begin{proof}
It follows by the Brouwer's invariance of domain theorem, see \cite{Bro1}. 
\end{proof}

\begin{lem}\label{lm:permutation} %%%maledetto LEMMA 1
Fix a subset $U\subseteq f[C]$. Then $f^{-1} \left[\mathrm{int}(U)\right]=\mathrm{int}(f^{-1}\left[U\right])$ and $f^{-1} \left[\mathrm{cl}(U)\right]=\mathrm{cl}(f^{-1}\left[U\right])$. 
\end{lem}
\begin{proof}
On the one hand, by \cite[Proposition 1.4.1]{Engelk}, the continuity of $f$ is equivalent to $f^{-1}[\mathrm{int}(V)] \subseteq \mathrm{int}(f^{-1} \left[V\right])$ for all $V\subseteq f[C]$. 
On the other hand, 
in view of Lemma \ref{lem:brower}, 
the restriction of the inverse $f^{-1}$ on $\mathrm{int}(U)$ is continuous and injective. 
This implies that $f[\mathrm{int}(f^{-1}[U])]$ is an open set contained in $f[f^{-1}[U]]=U$. 
Hence $f[\mathrm{int}(f^{-1}[U])] \subseteq \mathrm{int}(U)$ and we conclude that $\mathrm{int}(f^{-1}[U]) \subseteq f^{-1}[\mathrm{int}(U)]$. 

For the second part, set $V^c:=f[C] \setminus V$ for each $V\subseteq f[C]$. Then 
\begin{displaymath}
\begin{split}
f^{-1}\left[\mathrm{cl}(U)\right]
=f^{-1}\left[\left(\mathrm{int}(U^c)\right)^c\right]
&=\left(f^{-1}\left[\mathrm{int}(U^c)\right]\right)^c\\
&=\left(\mathrm{int}\left(\left(f^{-1}\left[U\right]\right)^c\right)\right)^c=\mathrm{cl}\left(f^{-1}[U]\right),
\end{split}
\end{displaymath}
which concludes the proof. 
\end{proof}

Let us recall the classical result of Carath\'{e}odory:
\begin{lem}\label{lem:caratheodory}\cite[p.73]{Rudin}
Let $V\subseteq \mathbb{R}^k$ be an arbitrary set.  
Then 
$$
\mathrm{conv}(V)=\bigcup_{U\subseteq V:\, |U|\le k+1} \mathrm{conv}(U).
$$
\end{lem}

We will need also the following characterization of interior of convex hulls:
\begin{lem}\label{lem:gustin} \cite{MR20800}  
Let $V\subseteq \mathbb{R}^k$ be an arbitrary set. Then 
$$
\mathrm{int}(\mathrm{conv}(V))=\bigcup_{U\subseteq V:\, |U|\le 2k}\mathrm{int}(\mathrm{conv}(U)).
$$
\end{lem}
Note that the constant $2k$ in the above lemma cannot be improved: indeed, if $V:=\{e_1,-e_1,\ldots,e_k,-e_k\}$, where $\{e_1,\ldots,e_k\}$ stands for the canonical basis of $\mathbb{R}^k$, then the zero vector belongs to $\mathrm{int}(\mathrm{conv}(V))$, but it does not belong to $\mathrm{int}(\mathrm{conv}(U))$ for all $U\subseteq V$ with $|U|=2k-1$.

\begin{lem}\label{lm:divisioncaratheodorynew2} 
The following inclusion holds\textup{:}
\begin{equation}\label{eq:claim2}
f^{-1}\left[\mathrm{int}\left(\mathrm{conv}(f[S])\right)\right]\subseteq 
\bigcup_{T\subseteq S\colon |T|\le 2k}\mathrm{cl}\left(f^{-1}\left[\mathrm{int}\left(\mathrm{conv}(f[T])\right)\right]\right).
\end{equation}
\end{lem}
\begin{proof}
Note that the claim is obvious if $f[S]$ is not $k$-dimensional. Hence, let us assume hereafter that $\mathrm{int}(\mathrm{conv}(f[S]))\neq \emptyset$. 
By the injectivity of $f$ and Lemma \ref{lem:gustin}, the left hand side of \eqref{eq:claim2} can be rewritten as
$$
f^{-1}\left[\bigcup_{U \subseteq S\colon |U|\le 2k}\mathrm{int}\left(\mathrm{conv}(f[U])\right)\right].
$$
Now, fix $T\subseteq S$ with $|T|\le 2k$ and note that $\mathrm{int}\left(\mathrm{conv}(f[T])\right)$ is an open set contained in $f[C]$, so by the continuity of $f$ the set $f^{-1}[\mathrm{int}\left(\mathrm{conv}(f[T])\right)]$ is open and contained in $C$. Thus $\mathrm{int}\left(\mathrm{conv}(f[T])\right) \subseteq f[C]$, therefore by Lemma \ref{lm:permutation} the right hand side of \eqref{eq:claim2} can be rewritten as
$$
f^{-1}\left[\bigcup_{T\subseteq S\colon |T|\le 2k}
\mathrm{cl}\left(\mathrm{int}\left(\mathrm{conv}(f[T])\right)\right)\right].
$$ 
At this point, the conclusion is immediate.
\end{proof}

\begin{lem}\label{lm:sufficiencyinterior}
Let $V\subseteq \mathbb{R}^k$ a be $k$-dimensional convex set. If $H\subseteq \mathbb{R}^k$ is relatively dense in $\mathrm{int}(V)$, then $H$ is relatively dense in $V$.
\end{lem}
\begin{proof}
It is sufficient to show that $\mathrm{int}(V)$ is relatively dense in $V$. 
For, let $U$ be an open ball centered at some boundary point of $V$, so that $U\cap V\neq \emptyset$, and let us say $u \in U\cap V$. 
Since $\mathrm{int}(V)\neq \emptyset$, there exists a closed ball $W$ with positive radius contained in $\mathrm{int}(V)$. If $u$ belongs to $\mathrm{int}(V)$, the claim holds. Otherwise $u$ has to be a boundary point of $V$. 
Now, since $V$ is convex, the subset $D:=\{\gamma u+(1-\gamma)w: \gamma \in [0,1], w \in W\}$ is contained in $V$. 
In addition, $u$ is an interior point of $U$, hence there exists an open ball $B$ centered in $u$ with $B\subseteq U$. By monotonicity, we obtain $\mathrm{int}(D) \cap B \subseteq \mathrm{int}(D) \cap U \subseteq \mathrm{int}(V) \cap U$. The claim follows from the fact the left hand side is clearly nonempty.
\end{proof}

Lastly, let us recall the following useful property: 
\begin{lem}\label{lem:intcl} \cite[Chapter II, Section 6, Proposition 16 and Corollaire 1]{MR633754} 
Let $U,V\subseteq \mathbb{R}^k$ be $k$-dimensional convex sets with $U$ open and $V$ closed. Then 
$$
\mathrm{int}(\mathrm{cl}(U))=U 
\quad \text{ and }\quad 
\mathrm{cl}(\mathrm{int}(V))=V.
$$
\end{lem}

\section{Proofs Section}\label{sec:proofs}

We proceed now to the proof of Theorem \ref{thm:main} and Corollary \ref{cor:greengustin}. 

\begin{proof}[Proof of Theorem \ref{thm:main}]
\ref{item:cond1} $\implies$ \ref{item:cond2}. 
Let us suppose on the contrary that condition \ref{item:cond1} holds, while \ref{item:cond2} does not. 
To simplify the notation, define 
$$
P:=\mathrm{int}\left(\mathrm{conv}(S)\right),\,  Q:=f^{-1}\left[\mathrm{int}\left(\mathrm{conv}(f[S])\right)\right], \text{ and }
R:=\mathrm{int}\left(\mathrm{conv}(f[S])\right).
$$
Hence $P\setminus Q$ is nonempty. Note all $P$, $Q$, and $R$ are open subsets of $\mathbb{R}^k$. 
\begin{claim}\label{cl:openball}
$P\setminus \mathrm{cl}(Q)$ has nonempty interior. 
\end{claim}
\begin{proof}
If \ref{item:cond2} does not hold then $f[P]$ is not contained in $R$. Let us suppose that $f[P]$ is contained in $\mathrm{cl}(R)$. Since $f[P]$ is open by Lemma \ref{lem:brower}, it follows that $f[P] \subseteq \mathrm{int}(\mathrm{cl}(R))$. 
Since the interior of the closure of an open convex set coincides with the set itself by Lemma \ref{lem:intcl}, it follows that $f[P] \subseteq R$, contradicting our assumption. Therefore $f[P]\setminus \mathrm{cl}(R)\neq \emptyset$.  

Let $g$ be the restriction of the inverse map $f^{-1}$ to $R$ and note that $R=f[Q]$. By Lemma \ref{lem:brower}, $g$ is a homeomorphism onto its image $Q$. It follows by Lemma \ref{lm:permutation} that $f[P]$ is not contained in $\mathrm{cl}(f[Q])=\mathrm{cl}(g^{-1}[Q])=g^{-1}[\mathrm{cl}(Q)]=f[\mathrm{cl}(Q)]$. By the injectivity of $f$ this is equivalent to saying that the open set $P\setminus\mathrm{cl}(Q)$ is nonempty, which proves the claim. 
\end{proof}

Let us suppose that $\mathscr{M}^\omega(S)$ is relatively dense in $\mathrm{conv}(S)$. Since $\mathscr{M}^\omega(S)\subseteq f^{-1}[\mathrm{conv}(f[S])]$ and $P\subseteq \mathrm{conv}(S)$, we obtain that $f^{-1}[\mathrm{conv}(f[S])]$ is relatively dense in $P$, that is, $V\cap f^{-1}[\mathrm{conv}(f[S])] \neq \emptyset$ for each nonempty open subset $V\subseteq P$. Taking into account Claim \ref{cl:openball}, we obtain 
\begin{equation}\label{eq:contradictionfirstpart}
(P\setminus \mathrm{cl}(Q)) \cap f^{-1}[\mathrm{conv}(f[S])]\neq \emptyset. 
\end{equation}
Since $S$ is $k$-dimensional and $f$ is an homeomorphism by Lemma \ref{lem:brower}, $f[S]$ is $k$-dimensional, so that $R$ is a nonempty open set. It follows by Lemma \ref{lm:permutation} and Lemma \ref{lem:intcl} that
$$
f^{-1}\left[\mathrm{conv}(f[S])\right]\subseteq f^{-1}\left[\mathrm{cl}(\mathrm{conv}(f[S]))\right]=f^{-1}[\mathrm{cl}(R)]=\mathrm{cl}(Q).
$$
However, this contradicts \eqref{eq:contradictionfirstpart}.

\medskip

\ref{item:cond2} $\implies$ \ref{item:cond1}. 
Since $S$ is $k$-dimensional, $|S| \ge k+1 \ge 2$. 
Let us a fix a subset $T\subseteq S$, let us say $T=\{t_1,\ldots,t_h\}$, with $2\le h\le 2k$. Let $\Delta$ be the $(h-1)$-dimensional simplex, that is, the set of vectors $\delta=(\delta_1,\ldots,\delta_h)\in [0,1]^h$ with $\sum_{i\le h} \delta_i=1$. Moreover, define $W:=\bigcup_n W_n$ where, for each integer $n\ge 1$, 
$$
W_n:=\left\{\gamma\in \Delta: f^{-1}\left(\gamma_1f(t_1)+\cdots+\gamma_h f(t_h)\right) \in \mathscr{M}^n(T)\right\}. 
$$
\begin{claim}\label{claim:denseness}
$W$ is dense in $\Delta$. 
\end{claim}
\begin{proof}
For each $V\subseteq \mathbb{R}^h$, let $\pi(V)$ be the projection of $V$ on its first $h-1$ coordinates, that is, 
$$
\pi(V):=\{x\in \mathbb{R}^{h-1}: (x_1,\ldots,x_{h-1},y) \in V \text{ for some }y \in \mathbb{R}\}.
$$
Since the vectors $\delta$ in $\Delta$ satisfy the linear constraint $\sum_{i\le h} \delta_i=1$, it is sufficient to show that $\pi(W)$ is relatively dense in $\pi(\Delta)=\{\delta \in [0,1]^{h-1}: \sum_{i<h}\delta_i \le 1\}$. 

Define $\alpha:=\max\{\alpha_1,\ldots,\alpha_m\}$. 
Consider the cube $J_0:=[0,1]^{h-1}$ and let us construct recursively a sequence $(\mathscr{J}_n)_{n\ge 1}$ of increasingly finer partitions  of $J_0$ into $2^{n(h-1)}$ parallelepipeds as it follows. 
Set for convenience $\mathscr{J}_0:=\{J_0\}$ and suppose that $\mathscr{J}_{n-1}$ has been defined for some $n\ge 1$. 
At the $n$th step, let us divide each parallelepiped $J \in \mathscr{J}_{n-1}$ with $h-1$ hyperplanes with dimension $h-2$, each one perpendicular to an edge and dividing the edge itself in two segments with lenghts proportional to $\alpha$ and $1-\alpha$. 
Therefore we obtain a partition $\mathscr{J}_n$ of the original cube $J_0$ into $2^{n(h-1)}$ parallelepipeds of dimension $h-1$, each one with volume $\le \alpha^{n(h-1)}$ and diagonals of length $\le \alpha^n \sqrt{h-1}$. Let $V_n$ be the set of vertices of these parallelepipeds which belong also to $\pi(\Delta)$, and note that $V_n\subseteq \pi(W_n)$. Since $V:=\bigcup_n V_n$ is contained in $\pi(W)$, we obtain 
\begin{displaymath}
\begin{split}
\limsup_{n\to \infty} \max_{x \in \pi(W_n)} \min_{y \in \pi(W_n), x\neq y}\|x-y\|&
\le \limsup_{n\to \infty} \max_{J \in \mathscr{J}_n} \max_{x,y \in J, x\neq y}\|x-y\| \\
&\le \limsup_{n\to\infty}\alpha^n \sqrt{h-1}=0.
\end{split}
\end{displaymath}
Therefore $V$ (and, hence, $\pi(W)$) is relatively dense in $\pi(\Delta)$.
\end{proof}

At this point, let $g_T: \Delta \to \mathbb{R}^k$ be the  function defined by 
\begin{displaymath}
\forall \delta \in \Delta, \quad 
g_T(\delta):=\delta_1f(t_1)+\cdots \delta_hf(t_h). 
\end{displaymath}  
Since $g_T$ is continuous with image $g_T[\Delta]$, it follows by Claim \ref{claim:denseness} that $g_T[W]$ is relatively dense in $g_T[\Delta]$ (this easy fact can found, e.g., in the proof of \cite[Theorem 1.4.10]{Engelk}). In particular, $g_T[W]$ is relatively dense in $\mathrm{int}(\mathrm{conv}(f[T]))$. Taking into account that $f$ is a homeomorphism, we obtain that $f^{-1}[g_T[W]]$ is relatively dense in $f^{-1}[\mathrm{int}(\mathrm{conv}(f[T]))]$. Since $\mathscr{M}^\omega(T)$ contains $f^{-1}[g_T[W]]$, it follows that 
\begin{equation}\label{eq:conclusionT}
\mathscr{M}^\omega(T) \,\text{ is relatively dense in }\, \mathrm{cl}(f^{-1}[\mathrm{int}(\mathrm{conv}(f[T]))]).
\end{equation}

Considering that \eqref{eq:conclusionT} holds for all subsets $T\subseteq S$ with $2\le |T|\le 2k$, it follows by monotonicity that 
$$
\mathscr{M}^\omega(S) \,\text{ is relatively dense in }\, \bigcup_{T\subseteq S: |T|\le 2k}\mathrm{cl}(f^{-1}[\mathrm{int}(\mathrm{conv}(f[T]))]). 
$$
In view of Lemma \ref{lm:divisioncaratheodorynew2}, we get
$$
\mathscr{M}^\omega(S) \,\text{ is relatively dense in }\, f^{-1}[\mathrm{int}(\mathrm{conv}(f[S]))]. 
$$ 
Hence, by condition \ref{item:cond2} and the injectivity of $f$, we have
$$
\mathscr{M}^\omega(S) \,\text{ is relatively dense in }\, \mathrm{int}(\mathrm{conv}(S)). 
$$
The conclusion follows by Lemma \ref{lm:sufficiencyinterior} and the hypothesis that the convex hull of $S$ has nonempty interior.

\medskip

\ref{item:cond3} $\implies$ \ref{item:cond2}. 
Condition \ref{item:cond3} implies that
\begin{displaymath}
f\left[\mathrm{int}(\mathrm{conv}(S))\right]\subseteq f\left[\mathrm{conv}(S)\right] \subseteq \mathrm{conv}(f[S]).
\end{displaymath}
Moreover, $f$ is an open map by Lemma \ref{lem:brower}, hence $f\left[\mathrm{int}(\mathrm{conv}(S))\right]$ is contained in $\mathrm{int}(\mathrm{conv}(f[S]))$, i.e., condition \ref{item:cond2} holds.

\bigskip

Lastly, suppose that $S$ is compact.

\smallskip

\ref{item:cond2} $\implies$ \ref{item:cond3}. 
Since $S$ is $k$-dimensional, $f[S]$ is $k$-dimensional as well by condition \ref{item:cond2}. 
Moreover, by Lemma \ref{lem:brower}, $f$ is a homeomorphism onto its image. 
Since both continuity and convex hull operator preserve compacteness (see e.g. \cite[Theorem 3.1.10]{Engelk} and 
\cite[Theorem 3.20(d)]{Rudin}, respectively), then $\mathrm{conv}(S)$, $\mathrm{conv}(f[S])$, and $f^{-1}[\mathrm{conv}(f[S])]$ are compact sets. 
In addition, by condition \ref{item:cond2} and the monotonicity of the closure operator, we have 
\begin{displaymath}
\mathrm{cl}\left(\mathrm{int}\left(\mathrm{conv}(S)\right)\right) \subseteq \mathrm{cl}\left(f^{-1}\left[\mathrm{int}\left(\mathrm{conv}(f[S])\right)\right]\right).
\end{displaymath}
Since a $k$-dimensional closed convex set coincides with the closure of its interior by Lemma \ref{lem:intcl}, we obtain by Lemma \ref{lm:permutation} that
\begin{displaymath}
\begin{split}
\mathrm{conv}(S)=
\mathrm{cl}\left(\mathrm{int}\left(\mathrm{conv}(S)\right)\right)
&\subseteq \mathrm{cl}\left(f^{-1}\left[\mathrm{int}\left(\mathrm{conv}(f[S])\right)\right]\right) \\
&=\mathrm{cl}\left(\mathrm{int}\left(f^{-1}\left[\mathrm{conv}(f[S])\right]\right)\right)=
f^{-1}\left[\mathrm{conv}(f[S])\right],
\end{split}
\end{displaymath}
which is equivalent to condition \ref{item:cond3} by the injectivity of $f$.
\end{proof}

\medskip

\begin{rmk}
   As it is clear from the proof above, the implication \ref{item:cond3} $\implies$ \ref{item:cond2} holds independently of the compactness assumption on $S$.
\end{rmk}

\medskip

\begin{proof}[Proof of Corollary \ref{cor:greengustin}]
The image of an interval of a continuous real-valued function of a real variable is also an interval. 
Moreover, every interval is a convex set. 
Since $\mathrm{conv}(S)$ is an interval of $\mathbb{R}$ and the functions $f_i$ are continuous injections, we obtain $f_i[\mathrm{conv}(S)]=\mathrm{conv}(f_i[S])$ for each $i \in \{1,\ldots,k\}$. 
However, each inverse function $f_i^{-1}$ is a continuous injection, hence $f_i[\mathrm{int}(\mathrm{conv}(S))]$ is an open subsets of $f_i[\mathrm{conv}(S)]$, cf. Lemma \ref{lem:brower}. 
Therefore 
\begin{displaymath}
f_i[\mathrm{int}(\mathrm{conv}(S))] \subseteq \mathrm{int}\left(f_i[\mathrm{conv}(S)]\right)=\mathrm{int}\left(\mathrm{conv}(f_i[S])\right)
\end{displaymath}
for each $i\in \{1,\ldots,k\}$. By \cite[Proposition 1.4.1]{Engelk}, the map $f_i^{-1}$ is continuous if and only if $f_i[\mathrm{int}(V)]\subseteq \mathrm{int}(f_i[V])$ for all $V\subseteq I$. Therefore, we obtain that
\begin{displaymath}
\begin{split}
f[\mathrm{int}(&\mathrm{conv}(S^k))]
=f\left[\mathrm{int}(\mathrm{conv}(S))^k\right]
=\prod_{i=1}^kf_i\left[\mathrm{int}(\mathrm{conv}(S))\right] \\
&\subseteq \prod_{i=1}^k\mathrm{int}\left(f_i\left[\mathrm{conv}(S)\right]\right)
=\mathrm{int}\left(f[(\mathrm{conv}(S))^k]\right)
=\mathrm{int}\left(f[\mathrm{conv}(S^k)]\right).
\end{split}
\end{displaymath}
The conclusion follows from Theorem \ref{thm:main}. 
\end{proof}

\section{Acknowledgements} 
I am grateful to Fabio Maccheroni, Massimo Marinacci, Simone Cerreia--Vioglio (Universit\'{a} Bocconi), and Giulio Principi (New York University) for useful comments. I also thank an anonynous referee for a careful reading of the manuscript. 

\section{Declarations}
The author confirms that no data have been used for the preparation of the manuscript. In addition, he declares that there is no conflict of interest and that no fundings were received. 

\bibliographystyle{amsplain}
%\bibliography{means}

\end{document}